\documentclass{amsart}
\usepackage{amsfonts}
\usepackage{amsmath}
\usepackage{amssymb}
\usepackage{amsthm}

\newtheorem{theorem}{Theorem}

\newtheorem{lemma}{Lemma}
\newtheorem{rem}{Remark}

\begin{document}
\author{Mark Pankov}
\title{Metric characterization of apartments in dual polar spaces}
\address{Department of Mathematics and Informatics, University of Warmia and Mazury,
{\. Z}olnierska 14A, 10-561 Olsztyn, Poland}
\email{pankov@matman.uwm.edu.pl}

\maketitle

\begin{abstract}
Let $\Pi$ be a polar space of rank $n$ and
let ${\mathcal G}_{k}(\Pi)$, $k\in \{0,\dots,n-1\}$
be the polar Grassmannian formed by $k$-dimensional singular subspaces of $\Pi$.
The corresponding Grassmann graph will be denoted by $\Gamma_{k}(\Pi)$.
We consider the polar Grassmannian ${\mathcal G}_{n-1}(\Pi)$
formed by maximal singular subspaces of $\Pi$ and
show that the image of every isometric embedding of the $n$-dimensional hypercube graph $H_{n}$
in $\Gamma_{n-1}(\Pi)$ is an apartment of ${\mathcal G}_{n-1}(\Pi)$.
This follows from a more general result (Theorem 2) concerning isometric embeddings of
$H_{m}$, $m\le n$ in $\Gamma_{n-1}(\Pi)$.
As an application, we classify all isometric embeddings of $\Gamma_{n-1}(\Pi)$
in $\Gamma_{n'-1}(\Pi')$, where $\Pi'$ is a polar space of rank $n'\ge n$ (Theorem 3).
\end{abstract}

\section{Introduction}
The problem discussed in this note was first considered in \cite{Pankov2}
and  motivated by the well-known metric characterization of apartments in Tits buildings (Theorem 1).
By \cite{Tits}, a {\it building} is a simplicial complex $\Delta$ containing a family of subcomplexes
called {\it apartments} and satisfying certain axioms.
One of the axioms says that all apartments are isomorphic to a certain Coxeter complex
--- the simplicial complex associated with a Coxeter system.
The diagram of this Coxeter system defines the type of the building $\Delta$.

Maximal simplices of $\Delta$, they are called {\it chambers},
have the same cardinal number $n$ (the rank of $\Delta$).
Two chambers are said to be {\it adjacent} if their intersection consists of $n-1$ vertices.
We write ${\rm Ch}(\Delta)$ for the set of all chambers and
denote by $\Gamma_{\rm ch}(\Delta)$ the graph
whose vertex set is ${\rm Ch}(\Delta)$ and whose edges are pairs of adjacent chambers.
Let ${\mathcal A}$ be the intersection of ${\rm Ch}(\Delta)$ with an apartment of $\Delta$
and let $\Gamma({\mathcal A})$ be the restriction of the graph $\Gamma_{\rm ch}(\Delta)$
to ${\mathcal A}$.

\begin{theorem}[\cite{Brown}, p. 90]
A subset of ${\rm Ch}(\Delta)$ is the intersection of ${\rm Ch}(\Delta)$
with an apartment of $\Delta$ if and only if it is
the image of an isometric embedding of $\Gamma({\mathcal A})$ in $\Gamma_{\rm ch}(\Delta)$.
\end{theorem}

The vertex set of $\Delta$ can be naturally decomposed in $n$ disjoint subsets called
{\it Grassmannians}:
the vertex set is labeled by the nodes of the associated diagram
(such a labeling is unique up to a permutation on the set of nodes) and
all vertices corresponding to the same node form a Grassmannian.
More general Grassmannians defined by parts of the diagram were investigated
in \cite{Pasini}.

Let ${\mathcal G}$ be a Grassmannian of $\Delta$.
We say that $a,b\in {\mathcal G}$ are {\it adjacent} if
there exists a simplex $P\in \Delta$ such that $P\cup\{a\}$ and $P\cup\{b\}$ both are
chambers;
in this case, the set of all $c\in {\mathcal G}$ such that $P\cup\{c\}$ is a chamber
will be called the {\it line} joining $a$ and $b$.
The Grassmannian  ${\mathcal G}$ together with the set of all such lines is a partial linear space;
it is called the {\it Grassmann space} corresponding to ${\mathcal G}$.
The associated {\it Grassmann graph} is the graph $\Gamma_{\mathcal G}$
whose vertex set is ${\mathcal G}$ and whose edges are pairs of adjacent vertices;
in other words, $\Gamma_{\mathcal G}$  is the collinearity graph of the Grassmann space.
It is well-known that this graph is connected.
The intersections of ${\mathcal G}$ with apartments of $\Delta$ are called
{\it apartments} of the Grassmannian ${\mathcal G}$.

In \cite{CS,CKS,Kasikova1,Kasikova2} apartments of some Grassmannians
were characterized in terms of the associated Grassmann spaces.
We interesting in a metric characterization of apartments in Grassmannians
similar to Theorem 1.

Every building of type $\textsf{A}_{n-1}$, $n\ge 4$ is
the flag complex of an $n$-dimensional vector space $V$ (over a division ring).
The Grassmannians of this building are the usual Grassmannians
${\mathcal G}_{k}(V)$, $k\in\{1,\dots,n-1\}$
formed by $k$-dimensional subspaces of $V$.
Two elements of ${\mathcal G}_{k}(V)$ are adjacent if their intersection is
$(k-1)$-dimensional. The associated Grassmann graph is denoted by $\Gamma_{k}(V)$.
If $k=1,n-1$ then any two distinct vertices of $\Gamma_{k}(V)$ are adjacent
and the corresponding  Grassmann space is the projective space $\Pi_{V}$ associated with $V$
or the dual projective space $\Pi^{*}_{V}$, respectively.
Every apartment of ${\mathcal G}_{k}(V)$  is defined by a certain base $B\subset V$:
it consists of all $k$-dimensional subspaces spanned by subsets of $B$.
For every $S,U\in {\mathcal G}_{k}(V)$ the distance $d(S,U)$
in $\Gamma_{k}(V)$ is equal to
$$k-\dim(S\cap U)$$
and all apartments of ${\mathcal G}_{k}(V)$ are the images of isometric embeddings of
the Johnson graph $J(n,k)$ in $\Gamma_{k}(V)$.
However, the image of every isometric embedding of $J(n,k)$ in $\Gamma_{k}(V)$
is an apartment of ${\mathcal G}_{k}(V)$ if and only if $n=2k$.
This follows from the classification
of isometric embeddings of Johnson graphs $J(l,m)$, $1<m<l-1$ in the Grassmann graph $\Gamma_{k}(V)$,
$1<k<n-1$ given in \cite{Pankov2}.

Every building of type $\textsf{C}_{n}$ is
the flag complex of a rank $n$ polar space $\Pi$, i.e.
it consists of all flags formed by singular subspaces of $\Pi$.
Apartments of this building are defined by frames of $\Pi$
and the associated Grassmannians are the polar Grassmannians
${\mathcal G}_{k}(\Pi)$, $k\in\{0,\dots,n-1\}$
consisting of $k$-dimensional singular subspaces of $\Pi$.
We restrict ourself to the Grassmannian ${\mathcal G}_{n-1}(\Pi)$ formed by maximal singular subspaces of $\Pi$.
Two elements of ${\mathcal G}_{n-1}(\Pi)$ are adjacent if their intersection is $(n-2)$-dimensional.
The corresponding  Grassmann graph is denoted by $\Gamma_{n-1}(\Pi)$.
The associated Grassmann space is known as the {\it dual polar space} of $\Pi$.
We show that apartments of ${\mathcal G}_{n-1}(\Pi)$
can be characterized as the images of isometric embeddings of
the $n$-dimensional hypercube graph $H_{n}$ in $\Gamma_{n-1}(\Pi)$.
This is a partial case of a more general result (Theorem 2) concerning isometric embeddings of $H_{m}$, $m\le n$
in $\Gamma_{n-1}(\Pi)$.
As an application,
we describe all isometric embeddings of $\Gamma_{n-1}(\Pi)$ in $\Gamma_{n'-1}(\Pi')$,
where $\Pi'$ is a polar space of rank $n'\ge n$ (Theorem 3).
This result generalizes classical Chow's theorem \cite{Chow}
on automorphisms of the graph $\Gamma_{n-1}(\Pi)$.

\section{Basics}

\subsection{Graph theory}
Let $\Gamma$ be a connected graph.
The {\it distance} $d(v,w)$ between two vertices $v,w\in \Gamma$ is defined
as the smallest number $i$ such that there exists a path of length $i$
(a path consisting of $i$ edges) between $v$ and $w$;
a path  connecting $v$ and $w$ is called a {\it geodesic} if it consists of
$d(v,w)$ edges.
The number
$$\max\{\;d(v,w)\;:\;v,w\in \Gamma\;\}$$
is called the {\it diameter} of $\Gamma$.
Suppose that it is finite.
Two vertices of $\Gamma$ are said to be {\it opposite}
if the distance between them is maximal
(is equal to the diameter).

An {\it isometric embedding} of a graph $\Gamma$ in a graph $\Gamma'$
is an injection of the vertex set of $\Gamma$ to the vertex set of $\Gamma'$
preserving the distance between vertices.
The existence of isometric embeddings of $\Gamma$ in $\Gamma'$
implies that the diameter of $\Gamma$ is not greater than  the diameter of $\Gamma'$.
An isometric embedding of $\Gamma$ in $\Gamma'$ is an isomorphism of $\Gamma$
to a subgraph of $\Gamma'$;
the converse fails,
an isomorphism of $\Gamma$ to a proper subgraph of $\Gamma'$
needs not to be an isometric embedding.
We refer \cite{Deza} for the general theory of isometric embeddings of graphs.

\subsection{Hypercube graphs}
Let
$$J:=\{1,\dots,n,-1,\dots,-n\}.$$
A subset $X\subset J$ is said to be {\it singular} if
$$i\in X\;\Longrightarrow\;-i\not\in X$$
for all $i\in I$.
Every maximal singular subset consists of $n$ elements and for every $i\in \{1, . . . , n\}$
it contains $i$ or $-i$.
The $n$-{\it dimensional  hypercube graph} $H_{n}$
can be defined as the graph whose vertices are maximal singular subsets of $J$
and two such subsets are adjacent (connected by an edge) if their intersection
consists of $n-1$ elements.
This graph is connected and for every maximal singular subsets $X,Y\subset J$
the distance $d(X,Y)$ in the graph $H_{n}$ is equal to
$$n-|X\cap Y|.$$
The diameter of $H_{n}$ is equal to $n$ and $X,Y$ are opposite vertices of $H_{n}$
if and only if $X\cap Y=\emptyset$.

\subsection{Partial linear spaces}
Let $P$ be a non-empty set and let ${\mathcal L}$ be a family of proper subsets of $P$.
Elements of $P$ and ${\mathcal L}$ will be called {\it points} and {\it lines}, respectively.
Two or more points are said to be {\it collinear} if there is a line containing all of them.
Suppose that the pair $\Pi=(P,{\mathcal L})$ is a {\it partial linear space},
i.e. the following axioms hold:
\begin{enumerate}
\item[$\bullet$]
each line contains at least two points,
\item[$\bullet$] every point belongs to a line,
\item[$\bullet$]
for any distinct collinear points $p,q\in P$ there is precisely one line containing them,
this line is denoted by $p\,q$.
\end{enumerate}
We say that $S\subset P$ is a {\it subspace} of $\Pi$
if for any distinct collinear points $p,q\in S$
the line $p\,q$ is contained in $S$.
A {\it singular} subspace is a subspace where any two points are collinear
(the empty set and a single point are singular subspaces).

For every subset $X\subset P$ the minimal subspace containing $X$
(the intersection of all subspaces containing $X$) is called {\it spanned} by $X$
and denoted by $\langle X\rangle$.
We say that $X$ is {\it independent} if $\langle X\rangle$ is not spanned by a proper subspace of $X$.

Let $S$ be a subspace of $\Pi$ (possible $S = P$).
An independent subset $X\subset S$ is called a {\it base} of $S$ if $\langle X\rangle=S$.
The {\it dimension} of $S$ is the smallest cardinality $\alpha$
such that $S$ has a base of cardinality $\alpha+1$.
The dimension of the empty set and a single point is equal to $-1$ and $0$ (respectively),
lines are $1$-dimensional subspaces.

Two partial linear spaces $\Pi=(P,{\mathcal L})$ and $\Pi'=(P',{\mathcal L}')$
are {\it isomorphic} if there exists a bijection $f:P\to P'$ such that $f({\mathcal L})={\mathcal L}'$;
this bijection is called a {\it collineation} of $\Pi$ to $\Pi'$.
We say that an injection of $P$ to $P'$ is an {\it embedding} of $\Pi$
in $\Pi'$ if it sends lines to subsets of lines such that distinct lines
go to subsets of distinct lines.

\subsection{Polar spaces}
Following \cite{BS},
we define a {\it polar space} as a partial linear space $\Pi=(P,{\mathcal L})$
satisfying the following axioms:
\begin{enumerate}
\item[$\bullet$] each line contains at least three points,
\item[$\bullet$] there is no point collinear with all points,
\item[$\bullet$] if $p\in P$ and $L\in {\mathcal L}$
then $p$ is collinear with one or all points of the line $L$,
\item[$\bullet$] any flag formed by singular subspaces is finite.
\end{enumerate}
If there exists a maximal singular subspace of $\Pi$ containing more than one line
then all maximal singular subspaces of $\Pi$ are projective spaces
of the same dimension $n\ge 2$;
the number $n+1$ is called the {\it rank} of $\Pi$.

The collinearity relation on $\Pi$ will be denoted by $\perp$:
we write $p\perp q$ if $p,q\in P$ are collinear and $p\not\perp q$ otherwise.
If $X,Y\subset P$ then $X\perp Y$ means that every point of $X$ is collinear with all points of $Y$.

\begin{lemma}\label{lemma0}
The following assertions are fulfilled:
\begin{enumerate}
\item[{\rm (1)}]
If $X\subset P$ and $X\perp X$
then the subspace $\langle X\rangle$ is  singular
and $p\perp X$ implies that $p\perp \langle X\rangle$.
\item[{\rm (2)}]
If $S$ is a maximal singular subspace of $\Pi$
then $p\perp S$ implies that $p\in S$.
\end{enumerate}
\end{lemma}

\begin{proof}
See, for example, Subsection 4.1.1 \cite{Pankov1}.
\end{proof}

\subsection{Dual polar spaces}
Let $\Pi=(P,{\mathcal L})$ be a polar space of rank $n\ge 3$.
For every $k\in\{0,1,\dots,n-1\}$ we denote by ${\mathcal G}_{k}(\Pi)$
the Grassmannian consisting of all $k$-dimensional singular subspaces of $\Pi$.
Then ${\mathcal G}_{n-1}(\Pi)$ is formed by maximal singular subspaces of $\Pi$.
Recall two elements of ${\mathcal G}_{n-1}(\Pi)$ are adjacent if their intersection
is $(n-2)$-dimensional and the associated Grassmann graph is denoted by $\Gamma_{n-1}(\Pi)$.

Let $M$ be an $m$-dimensional singular subspace of $\Pi$.
If $m<k$ then we write $[M\rangle_{k}$ for the set of all elements of ${\mathcal G}_{k}(\Pi)$
containing $M$. In the case when $m=n-2$, the subset $[M\rangle_{n-1}$ is called a {\it line}
of ${\mathcal G}_{n-1}(\Pi)$. The Grassmannian ${\mathcal G}_{n-1}(\Pi)$ together with the set
of all such lines is a partial linear space; it is called the {\it dual polar space} of $\Pi$.
Two distinct points of the dual polar space are collinear if and only if they are
adjacent elements of ${\mathcal G}_{n-1}(\Pi)$.
Note that every maximal singular subspace of the dual polar space is a line.

Let $M$ be as above.
If $m<n-2$ then $[M\rangle_{n-1}$ is a non-singular subspace of the dual polar space.
Subspaces of such type are called {\it parabolic} \cite{CKS}.
We will use the following fact: {\it the parabolic subspace $[M\rangle_{n-1}$ is isomorphic
to the dual polar space of a rank $n-m-1$ polar space.}

Consider $[M\rangle_{m+1}$.
A subset ${\mathcal X}\subset [M\rangle_{m+1}$
is called a {\it line} if there exists $N\in [M\rangle_{m+2}$
such that ${\mathcal X}$ consists of all elements of $[M\rangle_{m+1}$ contained in $N$.
Then $[M\rangle_{m+1}$ together with the set of all such lines is a polar space of
rank $n-m-1$ (Lemma 4.4 \cite{Pankov1}).
If ${\mathcal Y}\subset [M\rangle_{m+1}$ is a maximal singular subspace of this polar space
then there exists $S\in [M\rangle_{n-1}$
such that ${\mathcal Y}$ consists of all elements of $[M\rangle_{m+1}$ contained in $S$.
So, we can identify maximal singular subspaces of $[M\rangle_{m+1}$
with elements of $[M\rangle_{n-1}$.
This correspondence is a collineation  between
$[M\rangle_{n-1}$ and the dual polar space of $[M\rangle_{m+1}$
(Example 4.6 \cite{Pankov1}).

For every $S,U\in {\mathcal G}_{n-1}(\Pi)$ the distance $d(S,U)$
in the graph $\Gamma_{n-1}(\Pi)$ is equal to
$$n-1-\dim(S\cap U).$$
The diameter of $\Gamma_{n-1}(\Pi)$ is equal to $n$ and two vertices of $\Gamma_{n-1}(\Pi)$
are opposite if and only if they are disjoint elements of ${\mathcal G}_{n-1}(\Pi)$.

\section{Apartments of dual polar spaces}

\subsection{Main result}
Let $\Pi=(P,{\mathcal L})$ be a polar space of rank $n$.
Apartments of ${\mathcal G}_{k}(\Pi)$ are defined by frames of $\Pi$.
A subset $\{p_{1},\dots,p_{2n}\}\subset P$ is called a {\it frame}
if for every $i\in\{1\,\dots,2n\}$ there exists unique $\sigma(i)\in\{1\,\dots,2n\}$
such that $p_{i}\not\perp p_{\sigma(i)}$.
Frames are independent subsets of $\Pi$.
This guarantees that any $k$ mutually collinear points in a frame
span a $(k-1)$-dimensional singular subspace.

Let $B=\{p_{1},\dots,p_{2n}\}$ be a frame of $\Pi$.
The associated apartment ${\mathcal A}\subset {\mathcal G}_{n-1}(\Pi)$
is formed by all maximal singular subspaces spanned by subsets of $B$ ---
the subspaces of type $\langle p_{i_{1}},\dots, p_{i_{n}}\rangle$ such that
$$\{i_{1},\dots, i_{n}\}\cap \{\sigma(i_{1}),\dots, \sigma(i_{n})\}=\emptyset.$$
Thus every element of ${\mathcal A}$ contains precisely one of the points $p_{i}$ or $p_{\sigma(i)}$
for each $i$.
By Subsection 2.2, ${\mathcal A}$ is the image of an isometric embedding of
$H_{n}$ in $\Gamma_{n-1}(\Pi)$.

Let $M$ be an $(n-m-1)$-dimensional singular subspace of $\Pi$
and let $B$ be a frame of $\Pi$ such that $M$ is spanned by a subset of $B$.
The intersection of the associated apartment of ${\mathcal G}_{n-1}(\Pi)$
with the parabolic subspace $[M\rangle_{n-1}$ is called an {\it apartment} of $[M\rangle_{n-1}$.
This parabolic subspace can be identified with the dual polar space
of the rank $m$ polar space $[M\rangle_{n-m}$ (Subsection 2.5)
and every apartment of $[M\rangle_{n-1}$ is defined by a frame of this polar space
(\cite{Pankov1}, p. 180).
All apartments of $[M\rangle_{n-1}$ are the images of isometric embeddings of $H_{m}$ in $\Gamma_{n-1}(\Pi)$.

\begin{theorem}\label{theorem2}
The image of every isometric embedding of $H_{m}$, $m\le n$
in $\Gamma_{n-1}(\Pi)$ is an apartment in a parabolic subspace $[M\rangle_{n-1}$, where
$M$ is an $(n-m-1)$-dimensional singular subspace of $\Pi$.
In particular, the image of every isometric embedding of $H_{n}$
in $\Gamma_{n-1}(\Pi)$ is an apartment of ${\mathcal G}_{n-1}(\Pi)$.
\end{theorem}

\begin{rem}{\rm
There are no isometric embeddings of $H_{m}$ in $\Gamma_{n-1}(\Pi)$ if $m>n$
(in this case, the diameter of $H_{m}$ is greater than the diameter of $\Gamma_{n-1}(\Pi)$).
However, there exists a subset ${\mathcal X}\subset {\mathcal G}_{n-1}(\Pi)$
such that the restriction of the graph $\Gamma_{n-1}(\Pi)$ to ${\mathcal X}$ is isomorphic to $H_{n+1}$
\cite{CS} (Example 2, Section 9).
}\end{rem}

\subsection{Lemmas}
To prove Theorem \ref{theorem2} we will use the following lemmas.

\begin{lemma}\label{lemma1}
If $X_{0},X_{1},\dots,X_{m}$ is a geodesic in $\Gamma_{n-1}(\Pi)$
then
$$X_{0}\cap X_{m}\subset X_{i}$$
for every $i\in \{1,\dots,m-1\}$.
\end{lemma}

\begin{proof}
We prove induction by $i$ that $M:=X_{0}\cap X_{m}$
is contained in every $X_{i}$. The statement is trivial if $i=0$.
Suppose that $i\ge 1$ and $M\subset X_{i-1}$.

Since $X_{0},X_{1},\dots,X_{m}$ is a geodesic,
we have $$d(X_{i},X_{m})< d(X_{i-1},X_{m})$$
and, by the distance formula (Subsection 2.5),
$$\dim (X_{i}\cap X_{m})>\dim (X_{i-1}\cap X_{m}).$$
The latter implies the existence of a point
$$p\in (X_{i}\cap X_{m})\setminus X_{i-1}.$$
Then $X_{i}$ is spanned by the $(n-2)$-dimensional singular subspace
$X_{i-1}\cap X_{i}$ and the point $p$.
On the other hand,
$$(X_{i-1}\cap X_{i})\perp M$$
($M$ is contained in $X_{i-1}$ by inductive hypothesis)
and $p\perp M$ ($p$ and $M$ both are contained in $X_{m}$).
By the first part of Lemma \ref{lemma0}, $X_{i}\perp M$.
Since $X_{i}$ is a maximal singular subspace, the second part of Lemma \ref{lemma0}
guarantees that $M\subset X_{i}$.
\end{proof}

\begin{lemma}\label{lemma2}
In the hypercube graph $H_{m}$ for every vertex $v$ there is a unique vertex opposite to $v$.
If vertices $v,w\in H_{m}$ are opposite  then for every vertex
$u\in H_{m}$
there is a geodesic connecting $v$ with $w$ and passing throughout $u$.
\end{lemma}

\begin{proof}
An easy verification.
\end{proof}

\begin{lemma}\label{lemma3}
The image of every isometric embedding of $H_{m}$, $m\le n$ in $\Gamma_{n-1}(\Pi)$
is contained in a parabolic subspace $[M\rangle_{n-1}$,
where $M$ is an $(n-m-1)$-dimensional singular subspace of $\Pi$.
\end{lemma}

\begin{proof}
Let $f$ be an isometric embedding of $H_{m}$ in $\Gamma_{n-1}(\Pi)$.
We take any opposite vertices $v,w\in H_{m}$.
Then $$d(f(v),f(w))=m$$ and, by the distance formula (Subsection 2.5),
$$M:=f(v)\cap f(w)$$
is an $(n-m-1)$-dimensional singular subspace of $\Pi$.
Lemmas \ref{lemma1} and \ref{lemma2} show that
$M$ is contained in $f(u)$ for every $u\in H_{m}$.
\end{proof}

\subsection{Proof of Theorem \ref{theorem2}}
If $M$ is an $(n-m-1)$-dimensional singular subspace of $\Pi$
then $[M\rangle_{n-m}$ is a polar space of rank $m$ and $[M\rangle_{n-1}$ can be
identified with the associated dual polar space (Subsection 2.5);
moreover, every apartment of $[M\rangle_{n-1}$ is defined by a frame of
the polar space $[M\rangle_{n-m}$.
Therefore, by Lemma \ref{lemma3},
it is sufficient to prove Theorem \ref{theorem2} only in the case when $m=n$.

Let $\{p_{1},\dots,p_{2n}\}$ be a frame of $\Pi$.
Denote by ${\mathcal A}$  the associated apartment of ${\mathcal G}_{n-1}(\Pi)$.
The restriction of $\Gamma_{n-1}(\Pi)$ to ${\mathcal A}$ is isomorphic to $H_{n}$.
Suppose that $f:{\mathcal A}\to {\mathcal G}_{n-1}(\Pi)$
is an injection which induces an isometric embedding of
$H_{n}$ in $\Gamma_{n-1}(\Pi)$.
Let ${\mathcal X}$ be the image of $f$.

Each ${\mathcal A}\cap [p_{i}\rangle_{n-1}$ is an apartment in
the parabolic subspace $[p_{i}\rangle_{n-1}$.
Since $[p_{i}\rangle_{n-1}$ is the dual polar space of
the rank $n-1$ polar space $[p_{i}\rangle_{1}$ (Subsection 2.5),
the restriction of $\Gamma_{n-1}(\Pi)$ to ${\mathcal A}\cap [p_{i}\rangle_{n-1}$
is isomorphic to $H_{n-1}$.
Lemma \ref{lemma3} implies the existence of points $q_{1},\dots, q_{2n}$ such that
$$f({\mathcal A}\cap [p_{i}\rangle_{n-1})\subset{\mathcal X}\cap [q_{i}\rangle_{n-1}$$
for every $i$.

For every $X\in {\mathcal A}\setminus [p_{i}\rangle_{n-1}$
there exists $Y\in {\mathcal A}\cap [p_{i}\rangle_{n-1}$ disjoint from $X$.
We have
$$d(X,Y)=d(f(X),f(Y))=n.$$
Then $f(X)$ and $f(Y)$ are disjoint and $q_{i}\in f(Y)$.
This means that $q_{i}\not\in f(X)$ and $f(X)$ does not belong to ${\mathcal X}\cap[q_{i}\rangle_{n-1}$.
Since $f({\mathcal A})={\mathcal X}$,
we get the equality
$$f({\mathcal A}\cap [p_{i}\rangle_{n-1})={\mathcal X}\cap [q_{i}\rangle_{n-1}.$$
If $i\ne j$ then ${\mathcal A}\cap [p_{i}\rangle_{n-1}$ and ${\mathcal A}\cap [p_{j}\rangle_{n-1}$
are distinct subsets of ${\mathcal A}$ and their images
${\mathcal X}\cap [q_{i}\rangle_{n-1}$ and ${\mathcal X}\cap [q_{j}\rangle_{n-1}$ are distinct.

Therefore, $q_{i}\ne q_{j}$ if $i\ne j$. For every $X\in {\mathcal A}$
we have
$$p_{i}\in X\;\Longleftrightarrow\; q_{i}\in f(X).$$
Also, $q_{i}\perp q_{j}$ if $j\ne \sigma(i)$
(we take any $X\in {\mathcal A}$ which contains $p_{i}$ and $p_{j}$,
then $q_{i}$ and $q_{j}$ both belong to $f(X)$).

\begin{lemma}\label{lemma4}
For any $\{i_{1},\dots, i_{n}\}\subset \{1,\dots,2n\}$
satisfying
$$\{i_{1},\dots, i_{n}\}\cap \{\sigma(i_{1}),\dots,\sigma(i_{n})\}=\emptyset$$
$\langle q_{i_{1}},\dots q_{i_{n}}\rangle $ is a maximal singular subspace of $\Pi$
and
$$f(\langle p_{i_{1}},\dots, p_{i_{n}}\rangle)=\langle q_{i_{1}},\dots, q_{i_{n}}\rangle.$$
\end{lemma}

\begin{proof}
Suppose that $q_{i_{n}}$ belongs to
the singular subspace $\langle q_{i_{1}},\dots q_{i_{n-1}}\rangle$.
Let $X$ and $Y$ be the elements of ${\mathcal A}$ spanned by
$$p_{i_{1}},\dots, p_{i_{n-1}},p_{\sigma(i_{n})}\;\mbox{ and }\;
p_{\sigma(i_{1})},\dots, p_{\sigma(i_{n-1})},p_{i_{n}},$$
respectively.
These subspaces are disjoint and the same holds for $f(X)$ and $f(Y)$.
We have
$$\langle q_{i_{1}},\dots, q_{i_{n-1}}\rangle\subset f(X)\;
\mbox{ and }\; q_{i_{n}}\in f(Y)$$
which contradicts $q_{i_{n}}\in \langle q_{i_{1}},\dots, q_{i_{n-1}}\rangle$.

Therefore,
$q_{i_{1}},\dots, q_{i_{n}}$ form an independent subset of $\Pi$
and
$\langle q_{i_{1}},\dots, q_{i_{n}}\rangle$ is an element of ${\mathcal G}_{n-1}(\Pi)$.
Since $f(\langle p_{i_{1}},\dots, p_{i_{n}}\rangle)$ contains $q_{i_{1}},\dots, q_{i_{n}}$,
this subspace coincides with $\langle q_{i_{1}},\dots, q_{i_{n}}\rangle$.
\end{proof}

By Lemma \ref{lemma4}, every element of ${\mathcal X}$ is spanned by some
$q_{i_{1}},\dots, q_{i_{n}}$.
We need to show that the points $q_{1},\dots,q_{2n}$ form a frame of $\Pi$.
Since $q_{i}\perp q_{j}$ if $j\ne \sigma(i)$,
it is sufficient to establish that $q_{i}\not\perp q_{\sigma(i)}$ for all $i$.

Suppose that $q_{i}\perp q_{\sigma(i)}$ for a certain $i$.
Then $q_{i}\perp q_{j}$ for every $j\in \{1,\dots 2n\}$
and $q_{i}\perp X$ for all $X\in {\mathcal X}$.
Since ${\mathcal X}$ is formed by maximal singular subspaces of $\Pi$,
the second part of Lemma \ref{lemma0} implies that $q_{i}$ belongs to every element of ${\mathcal X}$.
Then the distance between any two elements of ${\mathcal X}$ is not greater
then $n-1$ which is impossible.

\section{Application of Theorem \ref{theorem2}}

\subsection{}
Let $\Pi=(P,{\mathcal L})$ and $\Pi'=(P',{\mathcal L}')$ be polar spaces of rank $n$ and $n'$,
respectively.
In this section Theorem \ref{theorem2} will be exploited to study isometric embeddings
of $\Gamma_{n-1}(\Pi)$ in $\Gamma_{n'-1}(\Pi')$.
The existence of such embeddings implies that $n$ (the diameter of $\Gamma_{n-1}(\Pi)$)
is not greater than $n'$ (the diameter of $\Gamma_{n-1}(\Pi')$).
So, we assume that $n\le n'$.

Suppose that $n=n'$.
Every mapping $f:P\to P'$ sending frames of $\Pi$ to frames of $\Pi'$
is an embedding of $\Pi$ in $\Pi'$ (Subsection 4.9.6 \cite{Pankov1});
moreover, for every singular subspace $S$ of $\Pi$
the subset $f(S)$ spans a singular subspace whose dimension is equal to the dimension of
$S$.
The mapping of ${\mathcal G}_{n-1}(\Pi)$ to ${\mathcal G}_{n-1}(\Pi')$
transferring every $S$ to $\langle f(S)\rangle$ is an isometric embedding
of $\Gamma_{n-1}(\Pi)$ in $\Gamma_{n-1}(\Pi')$.

Now consider the general case.
Let $M$ be an $(n'-n-1)$-dimensional singular subspace of $\Pi'$.
By Subsection 2.5, $[M\rangle_{n'-n}$ is a polar space of rank $n$
and $[M\rangle_{n'-1}$ can be identified with the associated dual polar space
(if $n=n'$ then $M=\emptyset$ and $[M\rangle_{n'-n}$ coincides with $P$).
As above, every frames preserving mapping of $\Pi$ to $[M\rangle_{n'-n}$
(a mapping which sends frames to frames)
induces an isometric embedding of $\Gamma_{n-1}(\Pi)$ in $\Gamma_{n'-1}(\Pi')$.

\begin{theorem}\label{theorem3}
Let $f:{\mathcal G}_{n-1}(\Pi)\to {\mathcal G}_{n'-1}(\Pi')$
be an isometric embedding of $\Gamma_{n-1}(\Pi)$ in $\Gamma_{n'-1}(\Pi')$.
There exists  an $(n'-n-1)$-dimensional singular subspace $M$ of $\Pi'$
such that the image of $f$ is contained in $[M\rangle_{n'-1}$
and $f$ is induced by a frames preserving mapping of $\Pi$ to $[M\rangle_{n'-n}$.
\end{theorem}

\begin{rem}{\rm
Theorem \ref{theorem3} generalizes classical Chow's theorem \cite{Chow}:
if $n=n'$ then every isomorphism of $\Gamma_{n-1}(\Pi)$ to $\Gamma_{n-1}(\Pi')$
is induced by a collineation of $\Pi$ to $\Pi'$.
In \cite{Chow} this theorem was proved for the dual polar spaces
of non-degenerate reflexive forms; but Chow's method works in
the general case (Section 4.6.4 \cite{Pankov1}).
Some interesting  results concerning adjacency preserving transformations of
symplectic dual polar spaces were established in \cite{Huang1,Huang2}.
}\end{rem}

\subsection{Proof of Theorem \ref{theorem3}}
We will use the following.

\begin{theorem}\label{theorem4}
Let $M$ be an $(n'-n-1)$-dimensional singular subspace of $\Pi'$.
Every  mapping
$f:{\mathcal G}_{n-1}(\Pi)\to [M\rangle_{n'-1}$
sending apartments of ${\mathcal G}_{n-1}(\Pi)$ to apartments of
$[M\rangle_{n'-1}$ is induced by a frames preserving mapping of $\Pi$ to $[M\rangle_{n'-n}$.
\end{theorem}

\begin{proof}
This follows from Theorem 4.17 \cite{Pankov1} (Subsection 4.9.6).
\end{proof}

Theorem \ref{theorem3} will be a consequence of Theorems
\ref{theorem2}, \ref{theorem4} and the following lemma.

\begin{lemma}\label{lemma5}
Let $f:{\mathcal G}_{n-1}(\Pi)\to {\mathcal G}_{n'-1}(\Pi')$
be an isometric embedding of $\Gamma_{n-1}(\Pi)$ in $\Gamma_{n'-1}(\Pi')$.
There exists  an $(n'-n-1)$-dimensional singular subspace $M$ of $\Pi'$
such that the image of $f$ is contained in $[M\rangle_{n'-1}$.
\end{lemma}

\begin{proof}
Let $X_{0}$ and $Y_{0}$ be opposite vertices of $\Gamma_{n-1}(\Pi)$.
Then
$$d(f(X_{0}),f(Y_{0}))=n$$
and, by the distance formula (Subsection 2.5),
$$M:=f(X_{0})\cap f(Y_{0})$$
is an $(n'-n-1)$-dimensional singular subspace of $\Pi'$.
Let $X\in {\mathcal G}_{n-1}(\Pi)$ and let
$$X_{0},X_{1},\dots,X_{m}=X$$
be a path in $\Gamma_{n-1}(\Pi)$ connecting $X_{0}$ with $X$.
We show that every $f(X_{i})$ belongs to $[M\rangle_{n'-1}$.

It is clear that $X_{1}$ is opposite to $Y_{0}$ or $d(X_{1},Y_{0})=n-1$.
In the second case, we take a geodesic of $\Gamma_{n-1}(\Pi)$
containing $X_{0},X_{1},Y_{0}$.
The mapping $f$ transfers it to a geodesic of $\Gamma_{n'-1}(\Pi')$
containing $f(X_{0}),f(X_{1}),f(Y_{0})$.
Lemma \ref{lemma1} guarantees that
$f(X_{1})\in [M\rangle_{n'-1}$.

In the case when $X_{1}$ is opposite to $Y_{0}$, we have
$$\dim(f(X_{1})\cap f(Y_{0}))=n'-n-1.$$
If the subspace $f(X_{1})\cap f(Y_{0})$ coincides with $M$ then  $M\subset f(X_{1})$.
If these subspaces are distinct then there exists a point
$$p\in (f(X_{1})\cap f(Y_{0}))\setminus M.$$
This point does not belong to $f(X_{0})\cap f(X_{1})$
(otherwise $p\in f(X_{0})\cap f(Y_{0})=M$ which is impossible).
Thus $f(X_{1})$ is spanned by the $(n'-2)$-dimensional singular subspace $f(X_{0})\cap f(X_{1})$
and the point $p$.
As in the proof of Lemma \ref{lemma1},
$$f(X_{0})\cap f(X_{1})\perp M\;\mbox{ and }\; p\perp M.$$
Then $M\perp f(X_{1})$ and $M\subset f(X_{1})$.

So, for every $X\in {\mathcal G}_{n-1}(\Pi)$
adjacent to $X_{0}$ we have $f(X)\in[M\rangle_{n'-1}$.
The same holds for every $X\in {\mathcal G}_{n-1}(\Pi)$  adjacent to
$Y_{0}$ (the proof is similar).

Now we establish the existence of $Y_{1}\in {\mathcal G}_{n-1}(\Pi)$
opposite to $X_{1}$ and satisfying $f(Y_{1})\in[M\rangle_{n'-1}$.

It was noted above that $X_{1}$ is  opposite to $Y_{0}$  or $d(X_{1},Y_{0})=n-1$.
If $X_{1}$ and $Y_{0}$ are opposite then $Y_{1}=Y_{0}$ is as required.
In the case when $d(X_{1},Y_{0})=n-1$,
the intersection of $X_{1}$ and $Y_{0}$ is a single point.
We take any $(n-2)$-dimensional singular subspace $U\subset Y_{0}$
which does not contain this point.
There exists a frame of $\Pi$ whose subsets span $X_{1}$ and $U$
(see, for example, Proposition 4.7 \cite{Pankov1}).
The associated apartment of ${\mathcal G}_{n-1}(\Pi)$
contains an element $Y_{1}$ such that $U\subset Y_{1}$ and $X_{1}\cap Y_{1}=\emptyset$.
It is clear that $Y_{0}$ and $Y_{1}$ are adjacent; hence $f(Y_{1})\in[M\rangle_{n'-1}$.

Since
$$d(f(X_{1}),f(Y_{1}))=n,$$
the subspace $f(X_{1})\cap f(Y_{1})$ is $(n'-n-1)$-dimensional.
On the other hand, $f(X_{1})$ and $f(Y_{1})$ both belong to $[M\rangle_{n'-1}$
and we get
$$f(X_{1})\cap f(Y_{1})=M.$$
We apply the arguments  given above to $X_{1},Y_{1},X_{2}$
instead of $X_{0},Y_{0},X_{1}$ and establish that $f(X_{2})\in[M\rangle_{n'-1}$.
Step by step, we show that each $f(X_{i})$ belongs to $[M\rangle_{n'-1}$.
\end{proof}

Let $f$ be, as above, an isometric embedding of $\Gamma_{n-1}(\Pi)$ in $\Gamma_{n'-1}(\Pi')$.
Lemma \ref{lemma5} implies the existence of an $(n'-n-1)$-dimensional singular
subspace $M$ of $\Pi'$ such that the image of $f$ is contained in $[M\rangle_{n'-1}$.
By Theorem \ref{theorem2},
$f$ transfers apartments of ${\mathcal G}_{n-1}(\Pi)$ to apartments of $[M\rangle_{n'-1}$.
Theorem \ref{theorem4} gives the claim.


\begin{thebibliography}{999}
\bibitem{BS}
Buekenhout F., Shult E. E., {\it On the foundations of polar geometry},
Geom. Dedicata 3(1974), 155–-170.

\bibitem{Brown}
Brown K., {\it Buildings}, Springer 1989.

\bibitem{Deza} Deza M., Laurent M., 
{\it Geometry of cuts and metrics}, Algorithms and Combinatorics 15, Springer 1997.



\bibitem{Chow} Chow W.L.,
{\it On the geometry of algebraic homogeneous spaces}, Ann. of Math.
50(1949), 32–-67.


\bibitem{CS}
Cooperstein B. N., Shult E. E.,
{\it Frames and bases of Lie incidence geometries}, J. Geom.,
60(1997), 17--46.

\bibitem{CKS}
Cooperstein B. N., Kasikova A., Shult E.E.,
{\it Witt-type Theorems for Grassmannians and Lie Incidence Geometries},
Adv. Geom. 5 (2005), 15--36.

\bibitem{Huang1} Huang W.-l.,
{\it Adjacency preserving mappings of invariant subspaces of a
null system}, Proc. Amer. Math. Soc. 128(2000), 2451-–2455.


\bibitem{Huang2} Huang W.-l.,
{\it Characterization of the transformation group of the space of a null system},
Results Math. 40(2001), 226-–232.

\bibitem{Kasikova1} Kasikova A.,
{\it Characterization of some subgraphs of point-collinearity graphs of building geometries},
Eur. J. Comb. 28 (2007), 1493--1529.

\bibitem{Kasikova2} Kasikova A.,
{\it Characterization of some subgraphs of point-collinearity graphs of building geometries}
II, Adv. Geom. 9 (2009), 45--84.

\bibitem{Pankov1}
Pankov M., {\it Grassmannians of classical buildings}, Algebra and
Discrete Math. Series 2, World Scientific, 2010.

\bibitem{Pankov2}
Pankov M., {\it Isometric embeddings of Johnson graphs in Grassmann graphs},
J. Algebraic Combin., accepted.

\bibitem{Pasini}
Pasini A., {\it Diagram geometries}, Clarendon press, Oxford 1994.

\bibitem{Tits} Tits J.,
{\it Buildings of spherical type and finite BN-pairs},
Lecture Notes in Mathematics 386,
Springer 1974.
\end{thebibliography}
\end{document}